\documentclass[a4paper,11pt]{article}

\usepackage{times}
\usepackage{color}
\usepackage{mathrsfs}
\usepackage{dsfont}
\usepackage{hyperref}
\usepackage{amsmath,amsthm,amsfonts,amssymb,bm, mathtools}
\usepackage{graphicx,psfrag, subfigure,epstopdf, caption}
\usepackage{makeidx}
\usepackage{listings}
\usepackage{titlesec}
\usepackage{graphicx}
\usepackage{pgfplots}
\usepackage[numbers,sort&compress]{natbib}
\usepackage{indentfirst}
\usepackage[T1]{fontenc}
\usepackage[utf8]{inputenc}
\usepackage{authblk}
\usepackage{tikz}

\theoremstyle{plain}
\newtheorem{thm}{Theorem}
\newtheorem{lem}{Lemma}

\title{The comparison of two Zagreb-Fermat eccentricity indices}

\author[a]{Xiangrui Pan \thanks{panr2358@gmail.com}}
\author[a]{Cheng Zeng\thanks{czeng@sdtbu.edu.cn (Corresponding author)}}
\author[a]{Longyu Li \thanks{815253894lly@gmail.com}}
\author[a]{Gengji Li \thanks{lgengji@gmail.com}}

\date{}
\affil[a]{{\small{School of Mathematics and Information Science, Shandong Technology and Business University, Yantai, Shandong Province, 264003, P. R. China}}}
\begin{document}

\maketitle
	
\centerline{\large{ABSTRACT}}
	
\vspace{0.2cm}
	In this paper, we focus on comparing the first and second Zagreb-Fermat eccentricity indices of graphs. We show that
	$$\frac{\sum_{uv\in E\left( G \right)}{\varepsilon _3\left( u \right) \varepsilon _3\left( v \right)}}{m\left( G \right)} \leq
	\frac{\sum_{u\in V\left( G \right)}{\varepsilon _{3}^{2}\left( u \right)}}{n\left( G \right)} $$
	holds for all acyclic and unicyclic graphs. Besides, we verify that the inequality may not be applied to graphs with at least two cycles.
	\vspace{0.2cm}
	
	\section{Introduction}
	In the fields of mathematical chemistry and graph theory, various graph invariants have been developed to characterize the structural properties of chemical compounds and complex networks, see Refs. \cite{todeschini2008handbook,latora2017complex,trinajstic2018chemical,brouwer2011spectra}.
	
	Suppose $G=(V(G),E(G))$ is a connected simple graph with node set $V(G)$ and edge set $E(G)$.
	Gutman and Trinajsti\'c proposed two degree-based topological indices, namely, the first and second Zagreb indices \cite{gutman1972graph}, which are defined by
	\begin{equation*}
		Z_1(G)=\sum_{u\in V(G)}\deg_G^2(u)
	\end{equation*}
	and
	\begin{equation*}
		Z_2(G)=\sum_{uv\in E(G)}\deg_G(u)\deg_G(v),
	\end{equation*}
	respectively. Here $\deg_G(u)$ is the degree of vertex $u$ in $G$. Then, Gutman, Ru\v s\v ci\'c, Trinajsti\'c and Wilcox \cite{gutman1975graph} elaborated Zagreb indices. In Ref. \cite{borovicanin2017bounds}, a survey of the most significant estimates of Zagreb indices has been introduced. Todeschini and Consonni \cite{todeschini2008handbook, consonni2009molecular} pointed out that the Zagreb indices and their variants have a wide range of applications in QSPR and QSAR models.
	
	We denote by $n=n(G)$ the number of vertices of $G$ and by $m=m(G)$ the number of its edges. Recently, a conjecture about the averages of two Zagreb indices has been studied, that is,
	\begin{equation*}
		\frac{Z_2(G)}{m(G)}\geq\frac{Z_1(G)}{n(G)}.
	\end{equation*}
	Although this conjecture has been disproved, it has been shown to hold for all graphs with the maximum degree at most four \cite{hansen2007comparing}, for acyclic graphs \cite{vukicevic2008comparing} and for unicyclic graphs \cite{horoldagva2009comparing}. To better describe the structural characteristics of graphs and chemical compounds, Zagreb index variants are continuously extended, such as improved Zagreb index, regularized Zagreb index, etc. \cite{zhou2009novel,gutman2013degree,gutman2018beyond}.
	
	Another important invariant is eccentricity $\varepsilon_2(u;G)$ of a vertex $u$ in $G$, namely the maximum distance from $u$ to other vertices in $G$, which is defined by $\varepsilon_2(u;G):=\max\{d_G(u,v):v\in V(G)\}$. In analogy with the first and the second Zagreb indices, the first, $E_1(G)$ and second, $E_2(G)$, Zagreb eccentricity indices \cite{indices2010note, ghorbani2012new}, an important class of graph indices, were proposed by replacing degrees by eccentricities of the vertices. Here the first and second Zagreb eccentricity indices are defined by
	$$E_1(G)=\sum_{u\in V\left( G \right)}{\varepsilon _{2}^{2}\left( u \right)},\ E_2(G)=\sum_{uv\in E\left( G \right)}{\varepsilon _2\left( u \right) \varepsilon _2\left( v \right)}.$$
	respectively. The two Zagreb eccentricity indices attract a large amount of interest. Some general mathematical properties of $E_1(G)$ and $E_2(G)$ have been investigated in \cite{2011On, Das2013Some}. Very recently, Zagreb eccentricity indices have been applied to fractals \cite{xu2023fractal} to show the structural properties of fractals.
	
	In \cite{indices2010note}, Vuki{\v c}evi\'c and Graovac compared $E_1(G)/n(G)$ with $E_2(G)/m(G)$. They showed that the inequality
	\begin{equation*}
		\frac{E_2(G)}{m(G)}\leq\frac{E_1(G)}{n(G)}
	\end{equation*}
	holds for all acyclic and unicyclic graphs and is not always valid for general graphs.  Based on the conclusions in \cite{indices2010note}, Qi and Du \cite{2017On} determined the minimum and maximum Zagreb eccentricity indices of trees, while Qi, Zhou and Li \cite{2017Zagreb} presented the minimum and maximum Zagreb eccentricity of unicycle graphs.

Given a vertex triplet $S=\{u,v,w\}$ of $V(G)$, the distance of $S$ is the minimum of the total distance from a vertex $\sigma\in V(G)$ to the three vertices in $S$, and we may write
$$d_G(S)=\mathcal{F}_G(u,v,w)=\min_{\sigma\in V(G)}\{d(\sigma,u)+d(\sigma,v)+d(\sigma,w)\}.$$
This concept in Euclidean space was first raised by Fermat \cite{1999Geometric} and then extended in graphs. So the distance $\mathcal{F}_G(u,v,w)$ is also called Fermat distance. Similarly, we call $\varepsilon_3(u;G)=\max_{v,w\in V(G)}\mathcal{F}_G(u,v,w)$ the Fermat eccentricity of $u$ in $G$.  For notational convenience, we sometimes use $\varepsilon _{3}\left( u \right)$ instead of $\varepsilon _{3}\left( u;G \right)$. Very recently, Fermat eccentricity and its average were first investigated by  Li, Yu and Klav\v{z}ar et al. \cite{li2021average, li2021average2}. They also extended this concept to Steiner $k$-eccentricity, see \cite{li2021steiner}. Fermat distance, Steiner distance, and some of their related variants have been widely applied to many fields, including operations research, VLSI design, optical and wireless communication networks \cite{du2008steiner}.  
	
We now introduce two modified Zagreb eccentricity indices based on Fermat eccentricity, that is, the first Zagreb-Fermat eccentricity index
	\begin{equation*}
		F_1(G)=\sum_{u\in V\left( G \right)}{\varepsilon _{3}^{2}\left( u \right)},
	\end{equation*}
and the second Zagreb-Fermat eccentricity index
	\begin{equation*}
		F_2(G)=\sum_{uv\in E\left( G \right)}{\varepsilon _3\left( u \right) \varepsilon _3\left( v \right)}.
	\end{equation*}
	
In this paper, we focus on comparison of the first and second Zagreb-Fermat eccentricity indices and show that inequality
\begin{equation}\label{eq:important}
		\frac{F_2(G)}{m\left( G \right) }\leq \frac{F_1(G)}{n\left( G \right)}
\end{equation}
holds for all acyclic and unicyclic graphs. Based on two counterexamples given in Section \ref{sec4}, we observe that the inequality is not always valid for multicyclic graphs.
\section{Notations and Preliminaries}
We first recall some notations and definitions. The degree of a vertex $u\in V(G)$, denoted by $\deg_G(u)$, is the number of adjacent vertices of $u$ in graph $G$. We call a vertex a leaf if its degree is $1$. If a vertex has degree at least $2$, then it is an internal vertex. Let $P_n$, $C_n$, $K_{1,n-1}$ be the $n$-vertex path, cycle and star. Define the radius and diameter of $G$ by $\operatorname{rad}(G):=\min_{u\in V(G)}\varepsilon_{2}(u)$ and  $\operatorname{diam}(G):=\max_{u\in V(G)}\varepsilon_{2}(u)$, respectively. A vertex that realizes $\operatorname{rad}(G)$ is called the central vertex. It is well known that tree $T$ contains one or two central vertices. A path in $G$ of length $\operatorname{diam}(G)$ is called the diametrical path. Note that every graph contains at least one diametrical path. A path joining a vertex and its most distance vertex is called a eccentric path.
	
We now focus on Fermat eccentricity. The vertex $\sigma$ that minimize the sum $d(\sigma,u)+d(\sigma,v)+d(\sigma,w)$ is called Fermat vertex of vertex set $\{u,v,w\}$. Meanwhile, we call the minimal spanning tree of vertex set $\{u,v,w\}$ the Fermat tree of $\{u,v,w\}$. Note that the Fermat vertex and the Fermat tree of a given vertex triplet in a graph may not be unique.  We call the vertices that realize $\varepsilon_3(u;G)$ the Fermat eccentric vertices of $u$ and the corresponding Fermat tree the Fermat eccentric $u$-tree. The following lemma shows that the difference of Fermat eccentricity of  two adjacent vertices can not exceed $1$.
\begin{lem}\label{lem7}
	For all $uv \in E\left(G\right)$, we have $$\left| \varepsilon _3\left( u \right) -\varepsilon _3\left( v \right) \right|\le 1.$$
\end{lem}
\begin{proof}
Note that $u$, $v$ are adjacent. Hence for any vertices $x$, $y$ in $G$, $|\mathcal{F}(u,x,y)-\mathcal{F}(v,x,y)|\leq1$ holds by the minimum definition of the Fermat distance.  Let $\omega_1$, $\omega_2$ be two vertices that realize $\varepsilon_3(u)$. It is obvious that
\begin{equation*}
			|\varepsilon_3(u)-\mathcal{F}(v,\omega_1,\omega_2)|\leq1.
\end{equation*}
		Similarly, suppose that $\varepsilon_3(v)=\mathcal{F}(v,\omega_3,\omega_4)$, we thus have
		\begin{equation*}
			|\varepsilon_3(v)-\mathcal{F}(u,\omega_3,\omega_4)|\leq1.
		\end{equation*}
		Then we can know that
		\begin{equation*}
			\varepsilon _3\left( u \right)  -1\le \mathcal{F} \left( v,\omega _1,\omega _2 \right) \le \varepsilon _3\left( v \right)
		\end{equation*}
		and
		\begin{equation*}
			\varepsilon _3\left( v \right)  -1\le \mathcal{F} \left( u,\omega _3,\omega _4 \right) \le \varepsilon _3\left( u \right).
		\end{equation*}
		We thus get the desired inequality
		$\left| \varepsilon _3\left( v \right) -\varepsilon _3\left( u \right) \right|\leq 1$.
	\end{proof}
	The following lemma is the corollary of Lemma 2.6 in \cite{li2021average} and will be implicitly applied in our paper.
	\begin{lem}
		Suppose $T$ is a tree and $u\in V(T)$. Then the Fermat eccentric $u$-tree contains the longest path starting from $u$. In other words, we can select the farthest vertex from $u$ as a Fermat eccentric vertex of $u$.
	\end{lem}
	In the rest of our paper, much operations on the diametrical path will be added. So for a given path $P_{d+1}=v_0v_1\cdots v_d$, we set path segment of $P_{d+1}$ by $P_{[p:q]}=v_p\cdots v_q$, where $0\leq p<q\leq d$.
	\section{Acyclic Graphs}
	Two basic properties about eccentric path are listed below:
	\begin{enumerate}
		\item If tree has one central vertex, then each eccentric path passes through the central vertex;
		\item If tree have two central vertices, then each eccentric path passes through these two vertices.
	\end{enumerate}
	The above properties ensure that for a tree $T$, all central vertices must be on a diametrical path. Let $P_{d+1}=v_0v_1\cdots v_d$ be a diametrical path of length $d$ which contains all central vertices of a given tree $T$. Now the tree $T$ can be depicted by Fig. \ref{Fig1} when $T$ has only one central vertex $c$, and Fig. \ref{Fig2} when $T$ has two central vertices $c_1$, $c_2$. Moreover, $P_{d+1}$ has one central vertex $c$ when $d$ is odd, and two central vertices $c_1$ and $c_2$ when $d$ is even. Let $T_{v_i}$ be the subtree with root vertex $v_i$, as shown in Figs. \ref{Fig1} and \ref{Fig2}, where $i=1,\ldots,d-1$. We denote $\ell_i$ as the maximal distance from $v_i$ to all leaves in $T_{v_i}$, where $i=1,\ldots,d-1$, that is, $\ell_i=\varepsilon_2(v_i;T_{v_i})$. We set $\ell=\max\{\ell_i:i = 1,\ldots,d-1\}$ and denote one subtree $T_{v_i}$ with $\ell_i=\ell$ by $T_\ell$.

	\begin{figure}[bpth!]
		\centering
		\includegraphics[width=0.7\textwidth]{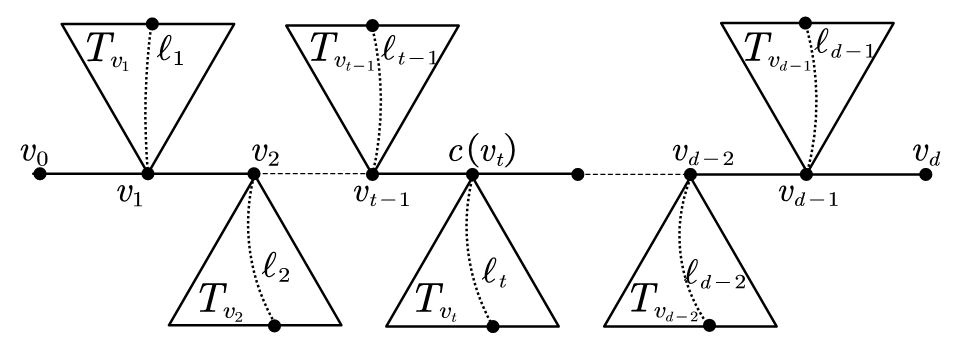}\\
		\caption{$G$ expanding at $P_{d+1}$ (One central vertex case).}\label{Fig1}
	\end{figure}

	\begin{figure}[bpth!]
		\centering
		\includegraphics[width=0.7\textwidth]{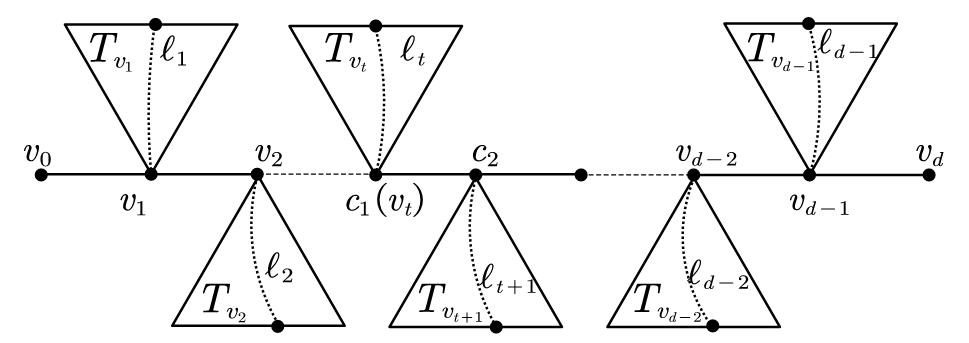}\\
		\caption{$G$ expanding at $P_{d+1}$ (Two central vertices case).}\label{Fig2}
	\end{figure}
	The following lemma is easy to prove since the diametrical path $P_{d+1}$ is the longest path in tree $T$.
	\begin{lem}\label{lem1}
		For all $i=1,\ldots, d-1$, inequality $\ell_i \leq \min\{i,d-i\}$ holds.
	\end{lem}
	Immediately from Lemma \ref{lem1}, we find the symmetry property of diametrical path, as established below:
	\begin{lem}
		For all $i=1,\ldots, d-1$, $\varepsilon_3\left ( v_i \right ) =\varepsilon_3\left ( v_{d-i} \right )$.
	\end{lem}
	\begin{proof}
		Symmetrically, we only need to check that the equation holds for $i=0,\cdots,\lfloor\frac{d}{2}\rfloor$. We obtain that
		\begin{align*}
			\varepsilon_3\left ( v_i \right ) &= d\left( v_i,v_d \right)+ \max_{k=i+1,\cdots,d-i-1}\left\{ \ell_k,i  \right\}
			\\
			&=d-i  +\max_{k=i+1,\cdots,d-i-1}\left\{ \ell_k,d-i  \right\}
			\\
			&=d\left( v_{d-i},v_0\right)+\max_{k=i+1,\cdots,d-i-1}\left\{ \ell_k,d-i  \right\}
			\\
			&=\varepsilon_3\left ( v_{d-i} \right ).
		\end{align*}
	\end{proof}
	
	From Lemma \ref{lem7}, we give more precise properties on the diametrical path $P_{d+1}$ of trees.
	\begin{lem}\label{lem3}
		For  $i\in \left\{ 0,\ldots,\lfloor \frac{d}{2} \rfloor \right\}$, Fermat eccentricity satisfies
		\begin{equation*}
			\varepsilon _3\left( v_{i+1} \right) \le \varepsilon _3\left( v_i \right) \le \varepsilon _3\left( v_{i+1} \right) +1.
		\end{equation*}
		For  $i\in \left\{ \lfloor \frac{d}{2} \rfloor+1 ,\ldots,d \right\}$, Fermat eccentricity satisfies
		\begin{equation*}
			\varepsilon _3\left( v_i \right) \le \varepsilon _3\left( v_{i+1} \right) \le \varepsilon _3\left( v_i \right) +1.
		\end{equation*}
		In other words, for vertices in $P_{d+1}$, the central vertices reach the minimum Fermat eccentricity.
	\end{lem}
	\begin{proof}
		For a given $v_i$, $i=0,\ldots,\left \lfloor \frac{d}{2}  \right \rfloor$, we obtain that
		\begin{align*}
			\varepsilon_3\left ( v_{i+1} \right ) &= d\left ( v_{i+1},v_d \right ) + \max_{k=i+2,\cdots,d-i-2}\left\{ \ell_k,i+1  \right\}
			\\
			&=d\left ( v_i,v_d \right )  +\max_{k=i+2,\cdots,d-i-2}\left\{ \ell_k-1,i  \right\}
			\\
			&\le d\left ( v_i,v_d \right )  +\max_{k=i+1,\cdots,d-i-1}\left\{ \ell_k,i  \right\} = \varepsilon_3\left ( v_i \right ).
		\end{align*}
		By Lemma \ref{lem1}, for $i=0,\ldots,\left \lfloor \frac{d}{2}  \right \rfloor-1$, we deduce that
		\begin{align*}
			&\ell_{i+1},\ell_{d-i-1} \le i+1,
		\end{align*}
		and hence,
		\begin{align*}
			\varepsilon_3\left ( v_{i+1} \right ) +1&= d\left ( v_i,v_d \right ) + \max_{k=i+2,\cdots,d-i-2}\left\{ \ell_k,i+1  \right\}
			\\
			&=d\left ( v_i,v_d \right )  +\max_{k=i+1,\cdots,d-i-1}\left\{ \ell_k,i+1  \right\}
			\\
			&\ge d\left ( v_i,v_d \right )  +\max_{k=i+1,\cdots,d-i-1}\left\{ \ell_k,i  \right\} = \varepsilon_3\left ( v_i \right ).
		\end{align*}
		By an analogous discussion, it is easy to prove the second assertion of Lemma \ref{lem3}.
		
		The formulas in Lemma \ref{lem3} explain that the Fermat eccentricity of vertex $v_i$ for $i\in \left \{ 0,\ldots,d \right \} $ is diminishing from two endpoints to the center of the diametrical path $P_{d+1}$, so central vertices reach the minimum Fermat eccentricity.
	\end{proof}
	\begin{lem}\label{lem5}
		(1) For all $uv\in E\left(P_{[\ell:d-\ell]}\right)$, it satisfies $\varepsilon _3(u)=\varepsilon _3(v)$;
		
		(2) For all $uv\in E\left(P_{[0:\ell]}\cup P_{[d-\ell:d]}\right)$, it satisfies $\left| \varepsilon _3(u)-\varepsilon _3(v) \right|=1$.
	\end{lem}
	\begin{proof}
		Suppose that $u,v\in \{v_0,\ldots,v_{\lfloor \frac{d}{2} \rfloor}\}$ and $u$ is more distant from the center. Let us proceed with the discussion by considering the following two cases:
		\begin{enumerate}
			\item $d\left ( v_0,u \right ) < \ell$.
			We have 	
			\begin{gather*}
				\varepsilon _3\left ( u \right )  = d\left ( u,v_d \right ) + \ell,
				\\
				\varepsilon _3\left ( v \right )  = d\left ( v,v_d \right ) + \ell =\varepsilon_3\left( u \right) - 1.
			\end{gather*}
			\item  $d\left ( v_0,u \right ) \ge  \ell$. We have 	
			\begin{gather*}
				\varepsilon _3\left ( u \right )  = \varepsilon _3\left ( v \right )= d.
			\end{gather*}
		\end{enumerate}
		For $u,v\in \left\{v_{\lfloor \frac{d}{2} \rfloor},\ldots,v_d\right\}$, we have the similar conclusion. Above all, when $uv\in E\left(P_{[\ell:d-\ell]}\right)$, we obtain $\varepsilon _3(u)=\varepsilon _3(v)$, and when $uv\in E\left(P_{[0:\ell]}\cup P_{[d-\ell:d]}\right)$, we obtain $\left| \varepsilon _3(u)-\varepsilon _3(v) \right|=1$. 	
	\end{proof}
	Lemma \ref{lem5} divides edges of $P_{d+1}$ into two parts and is the key lemma that will be used to prove Theorem \ref{thm1}. The following lemma determines the behaviors of subtree $T_{v_i}$, $i=1,2,\ldots, d-1$.
	\begin{lem}\label{lem4}
		For any vertex $u\in T_{v_i}$, $i=1,\ldots,d-1$, it holds that
		\begin{equation}\label{eq:lemma4}
			\varepsilon_3\left( u \right) = d\left( u,v_i \right)+ \varepsilon_3\left( v_i \right).
		\end{equation}
	\end{lem}
	\begin{proof}
		By Lemma \ref{lem1}, for any vertex $u \in T_{v_i}$, $i=1,\ldots,\lfloor \frac{d}{2} \rfloor$, it is obvious that one eccentric vertex of $u$ is $v_d$, and another one can be chosen as $v_0$ when $\ell_i=\ell$ and in subtree $T_{\ell}$ when $\ell>\ell_i$. In other words, we have
		\begin{align*}
			\varepsilon_3\left( u \right) = d\left( u,v_d \right)+\max_{t=i+1,\cdots,d-i-1}\left\{  \ell_t,i\right\}=d\left( u,v_i\right)+\varepsilon_3\left( v_i \right).
		\end{align*}
		Similarly, Eq. \eqref{eq:lemma4} holds for $i=\lceil \frac{d}{2} \rceil,\ldots,d-1 $.
	\end{proof}
	
	\begin{thm}\label{thm1}
		For any acyclic graph $T$, we have
		\begin{equation}\label{eq:thm1}
			\frac{\sum_{uv\in E\left( T\right)}{\varepsilon _3\left( u \right) \varepsilon _3\left( v \right)}}{m}\le \frac{\sum_{u\in V\left( T \right)}{\varepsilon _3^2\left( u \right)}}{n},
		\end{equation}
		where the equation holds if and only if $T\cong P_n$.
	\end{thm}
	\begin{proof}
		In this proof, we always assume that $u$ is further from the center than $v$ for any given edge $uv\in E(T)$. Whether the selected edge $uv$ is on the diametrical path $P_{d+1}$ leads to two cases as follows.
		\begin{enumerate}
			\item Edge $uv\in E\left( T_{v_k}\right)$, $k = 1,2,\ldots,d-1$.
			By Lemma \ref{lem4}, we have $\varepsilon _3\left( u \right) -\varepsilon _3\left( v \right) =1$ and thus
			\begin{align*}
				&\quad\frac{\sum_{uv\in E\left( T_{v_k} \right)}{\varepsilon _3\left( u \right) \varepsilon _3\left( v \right)}}{m}-\frac{\sum_{u\in V\left( T_{v_k}\setminus \{v_k\} \right)}{\varepsilon _{3}^{2}\left( u \right)}}{n}
				\\
				&=\frac{n\sum_{uv\in E\left( T_{v_k} \right)}{\varepsilon _3\left( u \right) \varepsilon _3\left( v \right)}-\left( n-1 \right) \sum_{u\in V\left( T_{v_k}\setminus \{v_k\} \right)}{\varepsilon _{3}^{2}\left( u \right)}}{n\left( n-1 \right)}
				\\
				&=\frac{n\sum_{u\in V\left( T_{v_k}\setminus \{v_k\} \right)}{\varepsilon _3\left( u \right) \left( \varepsilon _3\left( u \right) -1 \right)}-\left( n-1 \right) \sum_{u\in V\left( T_{v_k}\setminus \{v_k\} \right)}{\varepsilon _{3}^{2}\left( u \right)}}{n\left( n-1 \right)}\\
				&=\frac{\sum_{u\in V\left( T_{v_k}\setminus \{v_k\} \right)}{\varepsilon _3\left( u \right) \left( \varepsilon _3\left( u \right) -n \right)}}{n\left( n-1 \right)}<0.
			\end{align*}
			Hence,
			\begin{equation}
				\frac{\sum_{uv\in E\left( T_{v_k} \right)}{\varepsilon _3\left( u\right) \varepsilon _3\left( v \right)}}{m}\le \frac{\sum_{u\in V\left( T_{v_k}\setminus \{v_k\}\right)}{\varepsilon _3^2\left( u \right)}}{n}
			\end{equation}
			
			\item Edge $uv\in E\left( P_{d+1} \right)$. By Lemma \ref{lem5}, the difference can be rewritten by
			\begin{align*}
				L&=\frac{\sum_{uv\in E\left( P_{d+1} \right)}{\varepsilon _3\left( u \right) \varepsilon _3\left( v \right)}}{m}-\frac{\sum_{u\in V\left( P_{d+1} \right)}{\varepsilon _{3}^{2}\left( u \right)}}{n}
				\\
				&=\frac{n\sum_{uv\in E\left( P_{d+1} \right)}{\varepsilon _3\left( u \right) \varepsilon _3\left( v \right)}-\left( n-1 \right) \sum_{u\in V\left(P_{d+1} \right)}{\varepsilon _{3}^{2}\left( u \right)}}{n\left( n-1 \right)}
				\\
				&=\frac{n\sum_{uv\in E\left( P_{[\ell:d-\ell]} \right)}{\varepsilon _3\left( u \right) \varepsilon _3\left( v \right)}-\left( n-1 \right) \sum_{u\in V\left( P_{[\ell:d-\ell]} \right)}{\varepsilon _{3}^{2}\left( u \right)}}{n\left( n-1 \right)}
				\\
				&\quad+\frac{n\sum_{uv\in E\left(P_{[0:\ell]}\cup P_{[d-\ell:d]} \right)}{\varepsilon _3\left( u \right) \varepsilon _3\left( v \right)}-\left( n-1 \right) \sum_{u\in V\left( P_{[0:\ell]}\cup P_{[d-\ell:d]}\setminus\left\{v_\ell,v_{d-\ell}\right\} \right)}{\varepsilon _{3}^{2}\left( u \right)}}{n\left( n-1 \right)}.
			\end{align*}
			We consider the following two subcases. Notice that $\varepsilon_{3}(u)\leq n-1$ and $d+1\leq n$ always hold.
			\begin{enumerate}
				\item Acyclic graph $G$ has only one central vertex $c$.
				It derives that
				\begin{equation}\label{eq.2(a)}
					\begin{split}
						L&=\frac{n\sum_{u\in V\left( P_{[\ell:d-\ell]}\right) \setminus \left\{ c \right\}}{\varepsilon _{3}^{2}\left( u \right)}-\left( n-1 \right) \sum_{u\in V\left( P_{[\ell:d-\ell]}\right)}{\varepsilon _{3}^{2}\left( u \right)}}{n\left( n-1 \right)}
						\\
						&\quad +\frac{\splitdfrac{n\sum_{u\in V\left( P_{[0:\ell]}\cup P_{[d-\ell:d]}\setminus\left\{v_\ell,v_{d-\ell}\right\} \right)}{\varepsilon _3\left( u \right) \left( \varepsilon _3\left( u \right) -1 \right)}}{-\left( n-1 \right) \sum_{u\in V\left( P_{[0:\ell]}\cup P_{[d-\ell:d]}\setminus\left\{v_\ell,v_{d-\ell}\right\}\right)}{\varepsilon _{3}^{2}\left( u \right)}}}{n\left( n-1 \right)}
						\\
						& =\frac{\splitdfrac{\sum_{u\in V\left( P_{[\ell:d-\ell]} \right)}{\varepsilon _{3}^{2}\left( u \right)}-n\varepsilon _{3}^{2}\left( c \right) +\sum_{u\in V\left( P_{[0:\ell]}\cup P_{[d-\ell:d]}\setminus\left\{v_\ell,v_{d-\ell}\right\} \right)}{\varepsilon _{3}^{2}\left( u \right)}}{-n\sum_{u\in V\left( P_{[0:\ell]}\cup P_{[d-\ell:d]}\setminus\left\{v_\ell,v_{d-\ell}\right\} \right)}{\varepsilon _3\left( u \right)}}}{n\left( n-1 \right)}
						\\
						& =\frac{(d-2l+1)d^2-nd^2+\sum_{u\in V\left(P_{[0:\ell]}\cup P_{[d-\ell:d]}\setminus\left\{v_\ell,v_{d-\ell}\right\} \right)}{\varepsilon _3\left( u \right) \left( \varepsilon _3\left( u \right) -n \right)}}{n\left( n-1 \right)}\\
						&\le 0.
					\end{split}
				\end{equation}
				Hence,
				\begin{equation*}
					\frac{\sum_{uv\in E\left( P_{d+1} \right)}{\varepsilon _3\left( u \right) \varepsilon _3\left( v \right)}}{m}\le \frac{\sum_{u\in V\left( P_{d+1} \right)}{\varepsilon _3^2\left( u \right)}}{n}.
				\end{equation*}
				\item Acyclic graph $G$ has two central vertices $c_1$, $c_2$. We see that
				\begin{equation}\label{eq.2(b)}
					\begin{split}
						L&=\frac{n\sum_{u\in V\left( P_{[\ell:d-\ell]} \right) \setminus \left\{ c_1 \right\}}{\varepsilon _{3}^{2}\left( u \right)}-\left( n-1 \right) \sum_{u\in V\left( P_{[\ell:d-\ell]} \right)}{\varepsilon _{3}^{2}\left( u \right)}}{n\left( n-1 \right)}
						\\
						&\quad+\frac{\splitdfrac{n\sum_{u\in V\left(P_{[0:\ell]}\cup P_{[d-\ell:d]}\setminus\left\{v_\ell,v_{d-\ell}\right\} \right)}{\varepsilon _3\left( u \right) \left( \varepsilon _3\left( u \right) -1 \right)}}{-\left( n-1 \right) \sum_{u\in V\left( P_{[0:\ell]}\cup P_{[d-\ell:d]}\setminus\left\{v_\ell,v_{d-\ell}\right\} \right)}{\varepsilon _{3}^{2}\left( u \right)}}}{n\left( n-1 \right)}
						\\
						&=\frac{\splitdfrac{\sum_{u\in V\left( P_{[\ell:d-\ell]} \right)}{\varepsilon _{3}^{2}\left( u \right)}-n\varepsilon _{3}^{2}\left( c_1 \right)+\sum_{u\in V\left( P_{[0:\ell]}\cup P_{[d-\ell:d]}\setminus\left\{v_\ell,v_{d-\ell}\right\} \right)}{\varepsilon _{3}^{2}\left( u \right)}}{-n\sum_{u\in V\left(P_{[0:\ell]}\cup P_{[d-\ell:d]}\setminus\left\{v_\ell,v_{d-\ell}\right\} \right)}{\varepsilon _3\left( u \right)}}}{n\left( n-1 \right)}
						\\
						&=\frac{(d-2l+1)d^2-nd^2+\sum_{u\in V\left( P_{[0:\ell]}\cup P_{[d-\ell:d]}\setminus\left\{v_\ell,v_{d-\ell}\right\} \right)}{\varepsilon _3\left( u \right) \left( \varepsilon _3\left( u \right) -n \right)}}{n\left( n-1 \right)}\\
						&\le 0.
					\end{split}
				\end{equation}
				Hence,
				\begin{equation*}
					\frac{\sum_{uv\in E\left( P_{d+1}\right)}{\varepsilon _3\left( u \right) \varepsilon _3\left( v \right)}}{m}\le \frac{\sum_{u\in V\left( P_{d+1} \right)}{\varepsilon _3^2\left( u \right)}}{n}.
				\end{equation*}
			\end{enumerate}
		\end{enumerate}
		Combining above two cases together yields
		\begin{equation*}
			\frac{\sum_{uv\in E\left( T\right)}{\varepsilon _3\left( u \right) \varepsilon _3\left( v \right)}}{m}\le \frac{\sum_{u\in V\left( T \right)}{\varepsilon _3^2\left( u \right)}}{n}.
		\end{equation*}
		Suppose that $T\cong P_n$. We thus have $T_{v_k}=\{v_k\}$, $k=1,2,\ldots,d-1$. Consequently,
		\begin{equation*}
			\frac{\sum_{uv\in E\left( T \right)}{\varepsilon _3\left( u \right) \varepsilon _3\left( v \right)}}{n-1}-\frac{\sum_{u\in V\left( T \right)}{\varepsilon _{3}^{2}\left( u \right)}}{n}=\frac{\left( n-1 \right) d^2}{n-1}-\frac{nd^2}{n}=0.
		\end{equation*}
		On the other hand, if \eqref{eq:thm1} is an equation, then \eqref{eq.2(a)} and \eqref{eq.2(b)} must be equations. This means $(d-2\ell+1)d^2-nd^2=0$, and  we thus conclude $\ell=0$, which implies that $T$ is a path.
	\end{proof}
	Theorem \ref{thm1} partly reflects the fact that $F_1(T)$ and $F_2(T)$ may share some extremal bounds. We thus give the following theorem.
	\begin{thm}
		Suppose $T$ is a $n$-vertex tree, where $n\geq 3$. Then we have
		\begin{equation}\label{eq:thm21}
			F_1(K_{1,n-1})\leq F_1(T)\leq F_1(P_n)
		\end{equation}
		and
		\begin{equation}\label{eq:thm22}
			F_2(K_{1,n-1})\leq F_2(T)\leq F_2(P_n).
		\end{equation}
	\end{thm}
	\begin{proof}
		We first check the upper bounds of \eqref{eq:thm21} and \eqref{eq:thm22}. If $T\cong P_n$, then $\varepsilon_3(u)$ reach its maximum value $n-1$ for all $u\in V(T)$ and thus the right sides of \eqref{eq:thm21} and \eqref{eq:thm22} hold.
		
		For the lower bound, we claim that the left sides of \eqref{eq:thm21} and \eqref{eq:thm22} hold if the diameter of $T$ is $2$, equivalently, the maximal degree of $T$ is $n-2$.  Suppose that the diameter of $T$ is $d$ and $P_{d+1}=v_0v_1\cdots v_d$ is the diametrical path. Note that $v_0$ and $v_d$ are leaves. Then the vertices with maximal degree must be internal vertices. We choose one maximal degree vertex $v_i$ of $P_d$ and transform $T$ into another tree $T'$ by deleting edge $v_{d-1}v_d$ and attach $v_d$ to$v_i$ by an edge. From the definition of diameter and Lemma \ref{lem4}, we obtain $\varepsilon_{3}(u;T)\geq\varepsilon_{3}(u;T')$ for all vertices if we consider vertices and their transformations in pairs.
		We thus get
		$$F_1(T')-F_1(T)=\sum_{u\in V(T')}\varepsilon_{3}^2(u)-\sum_{u\in V(T)}\varepsilon_{3}^2(u)\leq 0$$
		and
		$$F_2(T')-F_2(T)=\sum_{uv\in E(T')}\varepsilon_{3}(u)\varepsilon_{3}(v)-\sum_{uv\in E(T)}\varepsilon_{3}(u)\varepsilon_{3}(v)\leq 0.$$
		We can repeat this transformation a sufficient number of times until arrive at a tree with maximal degree $n-2$, that is, we arrive at star $K_{1,n-1}$, which is the only graph with the smallest diameter $2$. Recall that the transformation does not increase the two Fermat-Zagreb indices. Therefore, we have $F_i(K_{1,n-1})\leq F_i(T)$, $i=1,2$.
	\end{proof}

\section{Unicyclic Graphs}
As shown in Fig. \ref{Fig3}, unicyclic graphs $G$ is consisting of the unique cycle $C_g$ with $g$ vertices and the maximum subtree $T_{x_i}$  with root vertex $x_i$ on $C_g$, where $i=1,\ldots ,g$.
\begin{figure}[bpth!]
		\centering
		\includegraphics[width=0.4\textwidth]{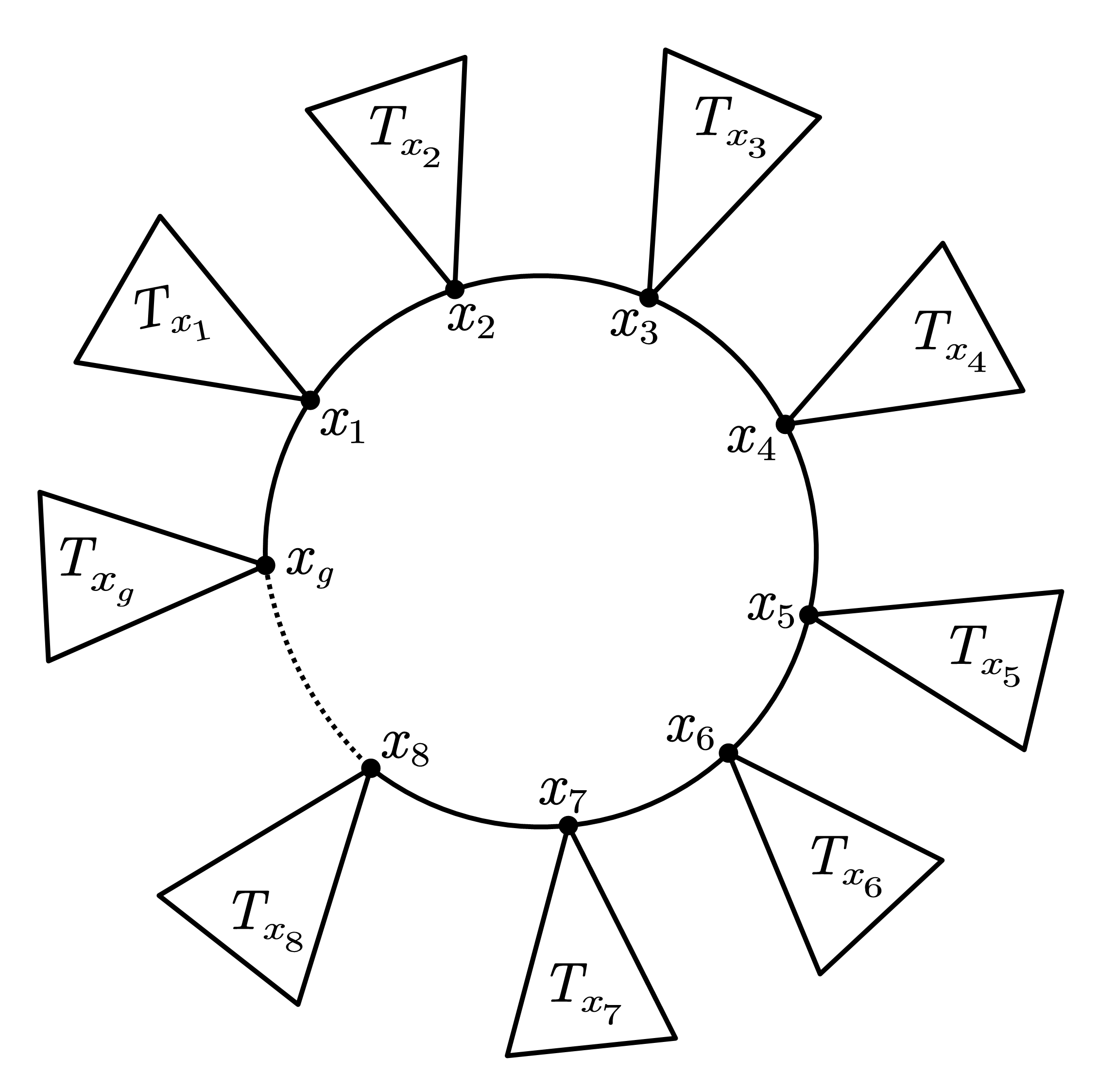}\\
		\caption{Unicyclic graph $G$.}\label{Fig3}
	\end{figure}
The following key lemma is from \cite{indices2010note}.
	\begin{lem}\label{lem6}
		Let $n\ge2$ and $x_1, \ldots, x_n$ $ $be positive integers such that $\left| x_i-x_{i+1}\right|\le1$ for each $i = 1,\ldots,n$, where $x_1=x_{n+1}$. Then $\sum_{i=1}^{n}{x_{i}^{2} }\ge \sum_{i=1}^{n}{x_ix_{i+1}}$.
	\end{lem}
	
	We are now in a position to prove the inequality for unicyclic graphs.
	\begin{thm}	
		When $G$ is a unicyclic graph, we have
		\begin{equation*}
			\frac{\sum_{uv\in E\left( G \right)}{\varepsilon _3\left( u \right) \varepsilon _3\left( v \right)}}{m}\le \frac{\sum_{u\in V\left( G \right)}{\varepsilon _{3}^{2}\left( u \right)}}{n}.
		\end{equation*}
	\end{thm}
	\begin{proof}
		We prove it by making a difference. Note that $n=m$ for unicyclic graphs. We thus have
		\begin{align*}
			\frac{\sum_{uv\in E\left( G \right)}{\varepsilon _3\left( u \right) \varepsilon _3\left( v \right)}}{m} - \frac{\sum_{u\in V\left( G \right)}{\varepsilon _{3}^{2}\left( u \right)}}{n}
			=\frac{\sum_{uv\in E\left( G \right)}{\varepsilon _3\left( u \right) \varepsilon _3\left( v \right)}-\sum_{u\in V\left( G \right)}{\varepsilon _{3}^{2}\left( u \right)}}{n}.
		\end{align*}
		It is equivalent to prove that
		\begin{equation*}
			\sum_{uv\in E\left( G \right)}{\varepsilon _3\left( u \right) \varepsilon _3\left( v \right)}-\sum_{u\in V\left( G \right)}{\varepsilon _{3}^{2}\left( u \right)}\le 0 .
		\end{equation*}
		Clearly,
		\begin{align*}
			&\quad\sum_{uv\in E\left( G \right)}{\varepsilon _3\left( u \right) \varepsilon _3\left( v \right)}-\sum_{u\in V\left( G \right)}{\varepsilon _{3}^{2}\left( u \right)}
			\\
			&=\left[ \sum_{uv\in E\left( C_g \right)}{\varepsilon _3\left( u \right) \varepsilon _3\left( v \right)}-\sum_{u\in V\left( C_g \right)}{\varepsilon _{3}^{2}\left( u \right)} \right]+\sum_{i=1}^g{\left[ \sum_{uv\in E\left( T_{x_i} \right)}{\varepsilon _3\left( u \right) \varepsilon _3\left( v \right)}-\sum_{u\in V\left( T_{x_i}\setminus \left\{ x_i \right\}  \right)}{\varepsilon _{3}^{2}\left( u \right)} \right]}.
		\end{align*}
		We can get that
		\begin{equation*}
			\sum_{uv\in E\left( C_g \right)}{\varepsilon _3\left( u \right) \varepsilon _3\left( v \right)}-\sum_{u\in V\left( C_g \right)}{\varepsilon _{3}^{2}\left( u \right)}\le 0.
		\end{equation*}
		by Lemmas \ref{lem7} and  \ref{lem6}.
		
		We now just need to prove
		\begin{equation*}
			\sum_{uv\in E\left( T_{x_i} \right)}{\varepsilon _3\left( u \right) \varepsilon _3\left( v \right)}-\sum_{u\in V\left( T_{x_i}\setminus \left\{ x_i \right\}  \right)}{\varepsilon _{3}^{2}\left( u \right)}\le 0
		\end{equation*}
		for any fixed $T_{x_i}$,  $i=1,\ldots,g$. For a given $x_i$, $i=1,\ldots,g$, we denote $x = x_i$ for convenient.  
		
		In the remaining part of this proof, suppose that $u$ is more distant from $x$ than $v$ for all $uv\in E(G)$. Let us now investigate $T_x $ in two cases.
		
		\noindent	\textbf{Case \uppercase\expandafter{\romannumeral1}:} $x$ is on one of the diametrical paths of $T_x $, see Fig. \ref{Fig5}.
		
		Suppose that $P_{d+1}=v_0v_1\cdots v_d$ is one diametrical path of $T_x$ which contains all central vertices of $T_x$. Set $c=v_t=v_{\lfloor \frac{d}{2} \rfloor}$, when $T_x$ has only one central vertex $c$ of $T_x$, and $c_1 =v_t=v_{\lfloor \frac{d}{2} \rfloor}$, when $T_x$ has two central vertices $c_1, c_2$ of $T_x$. Without loss of generality, let $x$ be on the left side of $c$, which means that $p\leq t$. We denote $v_q$ the symmetry vertex of $v_p$ about the center, i.e., $q = d - p $. Let $T_{v_j}$ be the subtree of $T_x$ with root vertex $v_j$. We also remark $\ell_j=\varepsilon_2(v_j; T_{v_j})$ and $\ell=\max\{\ell_j:j = 1,\ldots,d-1\}$.  Let $\omega'$, $\omega''$ be the vertices that realizes $\varepsilon_3(x;G\setminus T_x)$ and $T_\ell$ be one subtree $T_{v_j}$ with $\ell_j=\ell$. Denote one vertex in $T_\ell$ which is farthest from the root vertex of $T_\ell$ by $\bar{\omega}$. 
		
		\begin{figure}[bpth!]
			\centering
			\includegraphics[width=0.7\textwidth]{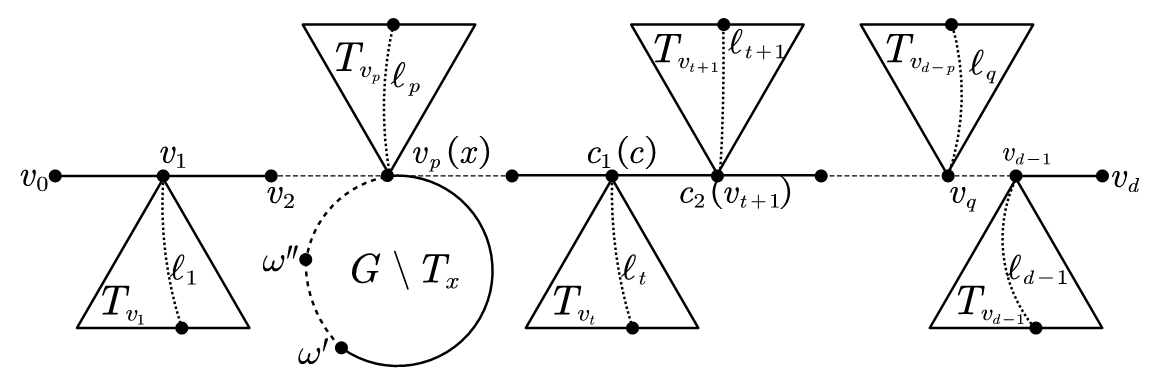}\\
			\caption{$T_x$ expanding at $P_{d+1}$.}\label{Fig5}
		\end{figure}
		
		\begin{enumerate}
			\item  $uv \in E\left(T_{v_j}\right)$, $j = 1,\ldots,d-1$. Recall that $\ell_j\leq\min\{j,d-j\}$. For all $w\in V\left(T_{v_j}\right)$, we can select the Fermat eccentric vertices of $\varepsilon_{3}(w)$ belonging to set $\{v_0,v_d,\omega', \omega'',\bar{\omega}\}$ when $\ell_j<\ell$, and $\{v_0,v_d,\omega',\omega''\}$ when $\ell_j=\ell$. In other words,  $\varepsilon_3(w)=d(w,v_j)+\varepsilon_3(v_j)$ and hence
			$\varepsilon _3\left( u\right)-\varepsilon _3\left( v\right) = 1$.

			So, for any $uv \in E\left(T_{v_j}\right)$, $j=1\ldots,d-1$, we observe that
			\begin{align*}
				&\quad\sum_{uv\in E\left( T_{v_j} \right)}{\varepsilon _3\left( u \right) \varepsilon _3\left( v \right)}-\sum_{u\in V\left( T_{v_j}\setminus\left\{v_j\right\} \right)}{\varepsilon _{3}^{2}\left( u \right)}
				\\
				&=\sum_{u\in V\left( T_{v_j}\setminus \left\{ v_j \right\} \right)}{\varepsilon _3\left( u \right) \left( \varepsilon _3\left( u \right) -1 \right)}-\sum_{u\in V\left( T_{v_j}\setminus\left\{v_j\right\} \right)}{\varepsilon _{3}^{2}\left( u \right)}
				\\
				&=-\sum_{u\in V\left( T_{v_j}\setminus \left\{ v_j \right\} \right)}{\varepsilon _3\left( u \right)}\le 0.
			\end{align*}
			
			\item $uv \in E\left(P_{[0:p]} \cup P_{[q:d]} \right) $.  For $w\in V\left(P_{[0:p]}\right)$, the farthest vertex from $w$ (also one Fermat eccentric vertex of $w$) belongs to $\{v_d,\omega', \omega''\}$. Similarly, for $w\in V\left(P_{[q:d]}\right)$, the farthest vertex from $w$ (also one Fermat eccentric vertex of $w$) belongs to $\{v_0,\omega', \omega''\}$. The two properties ensure that  $0\leq \varepsilon _3\left( u\right)-\varepsilon _3\left( v\right)\leq 1$. Hence,
			\begin{align*}
				&\quad\sum_{uv\in E\left( P_{[0:p]} \cup P_{[q:d]} \right)}{ \varepsilon _3\left( u \right)  \varepsilon _3\left( v \right)}-\sum_{u\in V\left(P_{[0:p]} \cup P_{[q:d]}\setminus\left\{v_p,v_q\right\} \right)}{ \varepsilon _{3}^{2}\left( u \right)}\\
				&\leq\sum_{u\in V\left( P_{[0:p]} \cup P_{[q:d]}\setminus\left\{v_p,v_q\right\} \right)}{ \varepsilon _{3}^{2}\left( u \right)}-\sum_{u\in V\left( P_{[0:p]} \cup P_{[q:d]}\setminus\left\{v_p,v_q\right\} \right)}{ \varepsilon _{3}^{2}\left( u \right)}=0.
			\end{align*}

			\item  $uv \in E\left(P_{[p:q]}\right) $. 
			
			\begin{enumerate}
				\item  For any $uv \in E\left(P_{[t:q]}\right)$, the farthest Fermat eccentric vertex of $u\in V\left(P_{[p:q]}\right)$ is in set $\{v_0,\omega',\omega''\}$. This property leads to
				$0 \le \varepsilon _3\left(u\right) -\varepsilon _3\left( v \right) \le 1$. Hence,
				\begin{align*}
					\sum_{uv\in E\left( P_{[t:q]} \right)}{\varepsilon _3\left( u \right) \varepsilon _3\left( v \right)}-\sum_{u\in V\left( P_{[t:q]}\setminus\left\{v_t\right\} \right)}{\varepsilon _{3}^{2}\left( u \right)} \le 0.
				\end{align*}
				
				\item For any $uv \in E\left(P_{[p:t]}\right)$, we have $-1\leq \varepsilon _3\left( u \right)- \varepsilon _3\left( v \right)\leq 1$ by Lemma \ref{lem7}.
				
				We need to check what occurs when $\varepsilon _3\left( u \right)- \varepsilon _3\left( v \right)=-1$. Recall that $u$ is more distant from $x$ than $v$. Suppose that there exists an edge $ v_kv_{k+1} \in E \left(P_{[p:t]} \right)$ such that $\varepsilon _3\left( v_{k}\right) - \varepsilon _3\left( v_{k+1} \right) = 1$. It is obvious that $\varepsilon _3\left( v_{k}\right)= \varepsilon _3\left( v_k; T_x \right)$ and
				$\varepsilon _3\left( v_{k+1}\right)= \varepsilon _3\left( v_{k+1}; T_x \right)$.
				Moreover, the Fermat eccentric vertices of $\varepsilon _3\left( v_{k}\right)$ and $\varepsilon _3\left( v_{k+1}\right)$ are both $v_d$ and $\bar{\omega}$. We thus have $\ell>\max\{k-p+\ell_p,k\}$ and then find that
				$$ \varepsilon _3\left( v_{d-k}\right)-\varepsilon _3\left( v_{d-k-1}\right)=1.$$
				
				Now we can take $v_kv_{k+1}$ and $v_{d-k-1}v_{d-k}$ into consideration pairwisely. It is clear that
				\begin{equation*}
					\begin{split}
						&\quad\varepsilon _3(v_k)\varepsilon _3(v_{k+1})-\varepsilon _3^2(v_{k+1})+\varepsilon _3(v_{d-k-1})\varepsilon _3(v_{d-k})-\varepsilon_3^2(v_{d-k})\\
						&=\varepsilon _3(v_{k+1})-\varepsilon _3(v_{d-k})=-1<0.
					\end{split}
				\end{equation*}
				
			\end{enumerate}
			So we get the inequation
			\begin{equation*}
				\sum_{uv\in E\left(P_{[p:q]} \right)}{ \varepsilon _3\left( u \right)  \varepsilon _3\left( v \right)}-\sum_{u\in V\left( P_{[p:q]} \setminus \left\{ x \right\} \right)}{ \varepsilon _{3}^{2}\left( u \right)} \le 0.
			\end{equation*}								
		\end{enumerate}
		Combining all three subcases in Case \uppercase\expandafter{\romannumeral1} together yields
		\begin{align*}
			\quad\sum_{uv\in E\left( T_x \right)}{ \varepsilon _3\left( u \right)  \varepsilon _3\left( v \right)}-\sum_{u\in V\left(T_x \setminus \left\{ x \right\} \right)}{ \varepsilon _{3}^{2}\left( u \right)} \le 0.
		\end{align*}	     	
		
		\noindent \textbf{Case \uppercase\expandafter{\romannumeral2}:} $x$ is not on any diametrical paths of the $T_x$. Then we can display $T_x$ in Fig. \ref{Fig6}. Set $P'=v_pv_{d+1}v_{d+2}\cdots v_{s+1}$, and then denote $T_{v_p} = \bigcup_{j=d+1}^{s+1}{T_{v_j}}\cup P'$. Some notations utilized in Case \uppercase\expandafter{\romannumeral1} are no longer reiterated.
		
		\begin{figure*}[htbp]
			\begin{center}
				\includegraphics[width=0.7\textwidth]{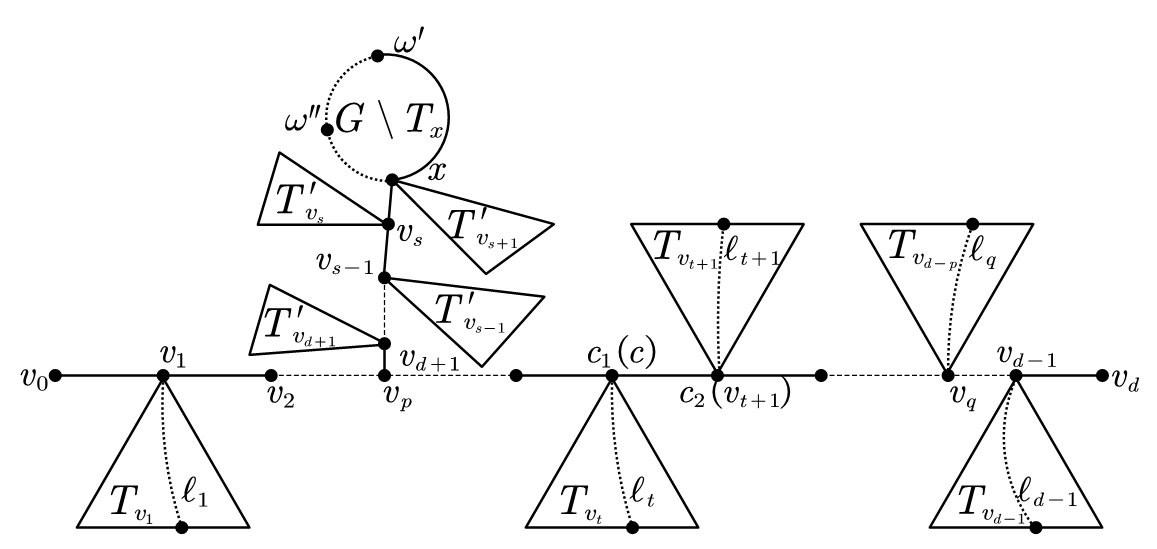}
			\end{center}
			\caption{$T_x$ expanding at $P_{d+1}$. Here let $x=v_{s+1}$ for notational convenience.}
			\label{Fig6}
		\end{figure*}
		\begin{enumerate}
			
			\item $uv\in E\left( T_x\setminus T_{v_p} \right)$.  We claim that
			\begin{align*}
				&\quad\sum_{uv\in E\left( T_x\setminus T_{v_p} \right)}{ \varepsilon _3\left( u \right)  \varepsilon _3\left( v \right)}-\sum_{u\in V\left( T_x\setminus T_{v_p} \right)}{ \varepsilon _{3}^{2}\left( u \right)} \le 0.
			\end{align*}
			Similar to the discussion in Case \uppercase\expandafter{\romannumeral1}, we have the following:
			\begin{itemize}
				\item[(1)] For $uv\in T_{v_j}$, $j=1,2,\ldots,p-1,p+1,\ldots d-1$, we have $\varepsilon_{3}(u)-\varepsilon_{3}(v)=1$.
				\item[(2)] For $uv\in E\left(P_{[0:p]} \cup P_{[q:d]} \right)$, we have $0\leq\varepsilon_{3}(u)-\varepsilon_{3}(v)\leq1$.
				\item[(3)] For $uv\in E\left(P_{[p:q]}\right)$, we focus on the edge $uv$ such that $\varepsilon_{3}(u)-\varepsilon_{3}(v)=-1$. Suppose that there exists an edge $v_kv_{k+1}$ such that
				$\varepsilon _3\left( v_{k}\right) - \varepsilon _3\left( v_{k+1} \right) = 1$, $k\in \{p,\ldots,t\}$.  We can still consider edge $v_{d-k-1}v_{d-k}$ with $\varepsilon _3\left( v_{d-k}\right)-\varepsilon _3\left( v_{d-k-1}\right)=1$, and then get $$\sum_{uv\in E\left(P_{[p:q]} \right)}{ \varepsilon _3\left( u \right)  \varepsilon _3\left( v \right)}-\sum_{u\in V\left( P_{[p:q]} \setminus \left\{ x \right\} \right)}{ \varepsilon _{3}^{2}\left( u \right)} \le 0.$$
			\end{itemize}
			In fact, the validity of the above three statements is obvious for the following reasons: In the discussion of  Case \uppercase\expandafter{\romannumeral1}, our focus solely lies on determining the precise location of Fermat eccentric vertices, without relying on the unicycle property of $G\backslash T_x$. This approach ensures the feasibility of our analysis in   Case \uppercase\expandafter{\romannumeral2} when substituting $G\backslash T_x$ and $T_{v_p}$ in Figure \ref{Fig5} with  $T_{v_p}\cup G\backslash T_x$ and single vertex $\{v_p\}$ in Figure \ref{Fig6}, respectively. So our claim holds.
			
			\item We now investigate edges in $E\left(T_{v_j}\right)$,  $j = d+1,\ldots ,s+1$. Recall that $x$ is not on any diametrical path. We have $\ell_p=\varepsilon_{2}(v_p;T_{v_p})< p$.  Furthermore, for any $uv \in E\left(T_{v_j}\right)$, $j = d+1,\ldots ,s+1$, the Fermat eccentric vertices of $u$ and $v$ do not belong to $V\left(T_{v_p}\right)$. Note that the distant between $u$ and $x$ is greater than that between $v$ and $x$.  Hence we have
			\begin{equation*}
				\varepsilon _3\left( u \right)=\varepsilon _3\left( v_j \right) +d\left( u,v_j \right)
			\end{equation*}
			and
			\begin{equation*}
				\varepsilon _3\left( v \right) =\varepsilon _3\left( v_j \right) +d\left( v,v_j \right) =\varepsilon _3\left( u \right) -1.
			\end{equation*}
			We derive that
			\begin{align*}
				&\sum_{uv\in E\left( T_{v_j} \right)}{\varepsilon _3\left( u \right) \varepsilon _3\left( v \right)}-\sum_{u\in V\left( T_{v_j}\setminus \left\{ v_j \right\} \right)}{\varepsilon _{3}^{2}\left( u \right)}
				\\
				&=\sum_{u\in V\left( T_{v_j}\setminus \left\{ v_j \right\} \right)}{\varepsilon _3\left( u \right) \left( \varepsilon _3\left( u \right) -1 \right)}-\sum_{u\in V\left( T_{v_j}\setminus \left\{ v_j \right\} \right)}{\varepsilon _{3}^{2}\left( u \right)}
				\\
				&= - \sum_{u\in V\left( T_{v_j}\setminus \left\{ v_j \right\} \right)}{\varepsilon _3\left( u \right)}\le 0
			\end{align*}
			for  $j = d+1,\ldots,s+1$.
			
			\item For $uv \in E \left(P'\right) $, note that $-1\leq\varepsilon _3(u)-\varepsilon _3(v)\leq 1$ from Lemma \ref{lem7}. Let us deal with the following two subcases:
			\begin{enumerate}
				\item For all $uv \in  E \left(P'\right)$, $0 \le \varepsilon _3\left( u\right) -\varepsilon _3\left( v \right) \le 1$.
				
				It is obvious that
				\begin{align*}
					&\sum_{uv\in E\left(P' \right)}{\varepsilon _3\left( u \right) \varepsilon _3\left( v \right)}-\sum_{u\in V\left( P'\setminus \left\{x \right\} \right)}{\varepsilon _{3}^{2}\left( u \right)} \le 0 .
				\end{align*}
				
				\item There exists an edge $v_kv_{k+1} \in E\left(P'\right)$ such that $\varepsilon _3\left( v_k\right) -\varepsilon _3\left( v_{k+1} \right) =-1$ and $k$ is maximum. Recall that $v_{k+1}$ is more distant from $x$ than $v_k$ and $\ell_p< \ell$. So the Fermat eccentric vertices of $v_k$ and $v_{k+1}$ must be $\omega'$ and $\omega''$. This statement is also true for every edge $uv$ of path  $v_pv_{d+1}\cdots v_{k+1}$, that is to say, $\varepsilon _3\left(u\right) -\varepsilon _3\left( v \right) =-1$. For any vertex $u$ in $P_{[0:p]}$,  it now becomes evident that the Fermat eccentric vertices of $u$ is also $\omega'$ and $\omega''$. We thus have
				$\varepsilon _3\left(u\right) -\varepsilon _3\left( v \right) =1$ for all $uv\in P_{[0:p]}$.
				
				Taking all the edges $uv\in E\left(P_{[0:k-d]}\right)$ and $P'$ into consideration, we have
				\begin{align*}
					&\quad\sum_{uv\in E(P'\cup P_{[0:k-d]})}\varepsilon _3\left(u\right)\varepsilon _3\left( v \right)- \sum_{u\in V((P_{[0:k-d]}\backslash \{v_{k-d}\})\cup (P'\backslash\{v_p\}))}\varepsilon _3^2\left(u\right)\\&=\sum_{u\in V(P'\backslash\{v_p\})}\varepsilon _3\left(u\right)-\sum_{u\in V((P_{[0:k-d]}\backslash \{v_{k-d}\})}\varepsilon _3(u)\leq 0,
				\end{align*}
				where the last inequality holds due to $\varepsilon _3(v_j)\geq \varepsilon _3(v_{k+1-j})$ for all $j=0,1,2,\ldots, k-d$.
			\end{enumerate}
			Note that although $\varepsilon_3(u)-\varepsilon_3(v)=-1$ may exist in both paths $P_{[p:t]}$ and $P'$. The vertex sets $V\left(P_{[p:q]}\right)$, $V\left(P'\right)$ and $V\left(P_{[0:k-d]}\right)$ we discussed are pairwise disjoint. Therefore, we conclude that
			\begin{align*}
				\sum_{uv\in E\left(T_x \right)}{\varepsilon _3\left( u \right) \varepsilon _3\left( v \right)}-\sum_{u\in V\left( T_x\setminus \left\{ x \right\} \right)}{\varepsilon _{3}^{2}\left( u \right)} \le 0.
			\end{align*}
		\end{enumerate}
		Above all, all cases are proved perspectively, and inequality \eqref{eq:important} holds for unicyclic graphs.
	\end{proof}
	
	\section{Multicyclic graphs}\label{sec4}
	In this section, we give two counterexamples which show that inequality \eqref{eq:important} does not always apply to general graphs.
	
	\begin{figure*}[htbp]
		\begin{center}
			\includegraphics[width=0.4\textwidth]{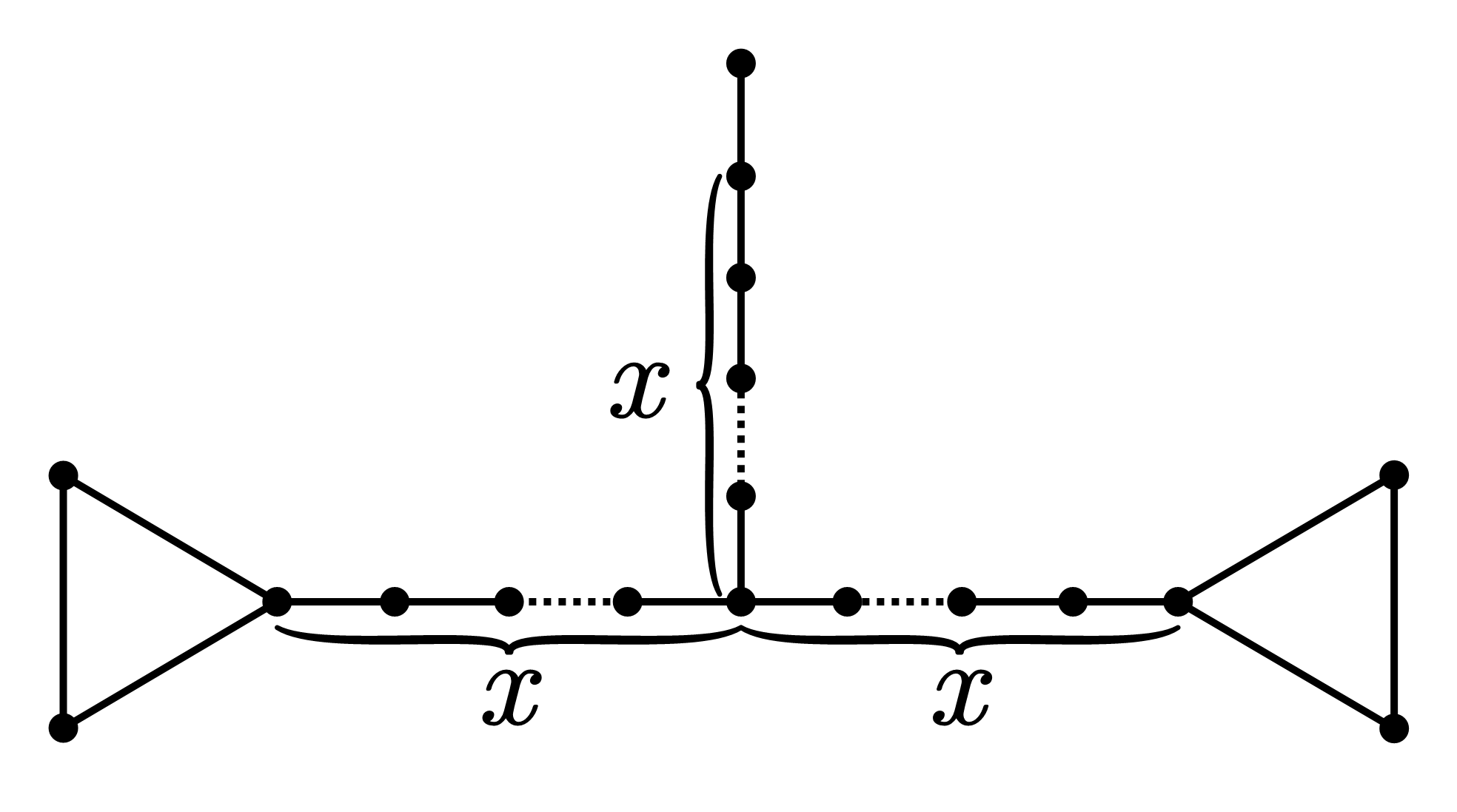}
		\end{center}
		\caption{Graph $\tilde{G}_{2,x}$.}
		\label{Fig7}
	\end{figure*}
	We give a bicyclic graph with $3x+6$ vertices shown in Fig. \ref{Fig7}. One can easily compute that
	\begin{equation*}
		\frac{\sum\limits_{u\in V(\tilde{G}_{2,x})}\varepsilon_3^2(u)}{n(\tilde{G}_{2,x})}-\frac{\sum\limits_{uv\in E(\tilde{G}_{2,x})}\varepsilon_3(u)\varepsilon_3(v)}{m(\tilde{G}_{2,x})}=\frac{-\frac{1}{2}x^3+31x^2+173x+55}{(3x+6)(3x+7)}.
	\end{equation*}
	
	We check that
	\begin{equation*}
		\frac{\sum\limits_{u\in V(\tilde{G}_{2,67})}\varepsilon_3^2(u)}{n(\tilde{G}_{2,67})}-\frac{\sum\limits_{uv\in E(\tilde{G}_{2,67})}\varepsilon_3(u)\varepsilon_3(v)}{m(\tilde{G}_{2,67})}>0
	\end{equation*}
	and
	\begin{equation*}
		\frac{\sum\limits_{u\in V(\tilde{G}_{2,68})}\varepsilon_3^2(u)}{n(\tilde{G}_{2,68})}-\frac{\sum\limits_{uv\in E(\tilde{G}_{2,68})}\varepsilon_3(u)\varepsilon_3(v)}{m(\tilde{G}_{2,68})}<0.
	\end{equation*}
	
	We next consider a graph with multiple cycles, as shown in Fig. \ref{Fig8}.
	\begin{figure*}[htbp]
		\begin{center}
			\includegraphics[width=0.4\textwidth]{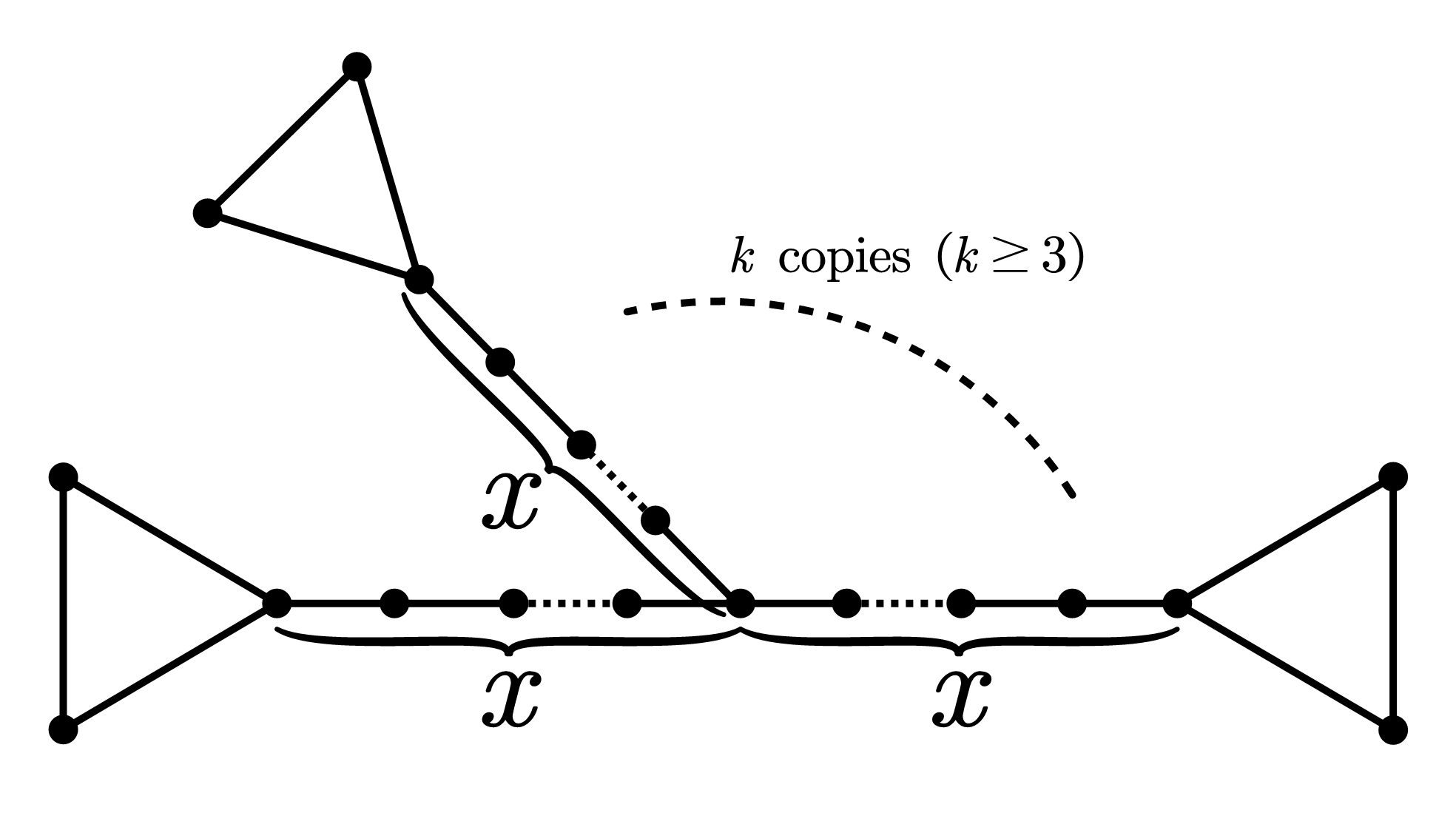}
		\end{center}
		\caption{Graph $G_{k,x}$ ($k\geq 3$).}
		\label{Fig8}
	\end{figure*}
	
	Without difficulty, we can get the following equation:
	\begin{align*}
		&\frac{\sum\limits_{u\in V(G_{k,x})}\varepsilon_3^2(u)}{n(G_{k,x})}-\frac{\sum\limits_{uv\in E(G_{k,x})}\varepsilon_3(u)\varepsilon_3(v)}{m(G_{k,x})}=\\
		& \quad\frac{\left(-\frac{1}{6}k^2-\frac{26k}{3}\right)x^3+(8k^2-54k)x^2+\left(\frac{121k^2}{6}-\frac{226k}{3}\right)x+12k^2-30k}{(kx+3k)(kx+2k+1)}.
	\end{align*}
	
	It is clear that
	$$\frac{\sum\limits_{u\in V(G_{k,x})}\varepsilon_3^2(u)}{n(G_{k,x})}-\frac{\sum\limits_{uv\in E(G_{k,x})}\varepsilon_3(u)\varepsilon_3(v)}{m(G_{k,x})}>0$$
	when $x=0$, while
	$$\frac{\sum\limits_{u\in V(G_{k,x})}\varepsilon_3^2(u)}{n(G_{k,x})}-\frac{\sum\limits_{uv\in E(G_{k,x})}\varepsilon_3(u)\varepsilon_3(v)}{m(G_{k,x})}<0$$ for $x$ large enough.
	
	The above two counterexamples demonstrate that inequality \eqref{eq:important} and its opposite inequality are not always valid for graphs with at least two cycles.
	\section{Conclusion}
	In this paper, we have investigated inequality
	\begin{equation*}
		\frac{\sum_{u\in V\left( G \right)}{\varepsilon _3^2\left( u \right)}}{n(G)}\geq\frac{\sum_{uv\in E\left( G \right)}{\varepsilon _3\left( u \right) \varepsilon _3\left( v \right)}}{m(G)}
	\end{equation*}
	and proved that this inequality holds for all acyclic and unicycle graphs. We negate the validity of this inequality and its opposite inequality in graphs with at least two cycles. The properties and applications of Zagreb-Fermat indices, or more generally, Zagreb Steiner $k$-indices, need further study. Zagreb-type indices are important in graphs, so a lot of extensive research and development still needs to be done.
	
	\section*{Acknowledgments}
	The research is partially supported by the Natural Science Foundation of China (No. 12301107) and the Natural Science Foundation of Shandong Province, China (No. ZR202209010046).
	\section*{Declaration of interest statement}
	All authors disclosed no relevant relationships.
	
	\bibliographystyle{unsrt}
	\bibliography{Ref}
\end{document}